\documentclass[11pt]{article}
\usepackage{amssymb,latexsym,amsmath,amstext,amsthm} 
\usepackage{url}
\usepackage{fullpage}
\usepackage{authblk}
\newcommand{\eff}{\ensuremath{\mathbb{F}}}
\newcommand{\zed}{\ensuremath{\mathbb{Z}}}
\renewcommand{\S}{\ensuremath{\mathcal{S}}}
\newcommand{\K}{\ensuremath{\mathcal{K}}}
\newcommand{\A}{\ensuremath{\mathcal{A}}}
\newcommand{\E}{\ensuremath{\mathcal{E}}}
\newcommand{\G}{\ensuremath{\mathcal{G}}}
\newcommand{\TT}{\ensuremath{\mathcal{T}}}

\newcommand{\D}{\ensuremath{\mathcal{D}}}

\newcommand{\prob}{\ensuremath{\mathbf{Pr}}}

\newcommand{\good}{\ensuremath{\mathsf{Good}}}

\newcommand{\he}{\ensuremath{\hat{\epsilon}}}

\newtheorem{theorem}{Theorem}[section]
\newtheorem{lemma}[theorem]{Lemma}
\newtheorem{corollary}[theorem]{Corollary}
\newtheorem{example}{Example}[section]
\newtheorem{definition}{Definition}[section]

\begin{document}

\title{Combinatorial Characterizations of Algebraic Manipulation Detection Codes Involving Generalized Difference Families}
\author[1]{Maura~B.~Paterson}
\author[2]{Douglas~R.~Stinson\thanks{D.~Stinson's research is supported by NSERC discovery grant 203114-11.}}
\affil[1]{Department of Economics,
Mathematics and Statistics\\ Birkbeck, University of London,
Malet Street, London WC1E 7HX, UK}
\affil[2]{David R.\ Cheriton School of Computer Science\\ University of Waterloo,
Waterloo, Ontario, N2L 3G1, Canada}

\date{\today}

\maketitle

\begin{abstract}This paper provides a mathematical analysis of {optimal}
algebraic manipulation detection (AMD) codes. We prove several lower bounds on the success probability
of an adversary and we then give some combinatorial characterizations of AMD codes that meet the bounds with equality. These characterizations involve various types of generalized difference families. Constructing
these difference families is an interesting problem in its own right.
\end{abstract}

\section{Introduction}
\label{notation.sec}

Algebraic manipulation detection (AMD) codes were defined in 
2008 by Cramer {\it et al.} \cite{CDFPW1,CDFPW2} as a generalization and abstraction of
techniques that were previously used in the study of robust secret sharing schemes \cite{OKSS,OKS,TW}. 
AMD codes are studied further in \cite{AS,CFP,CPX}. Several interesting and useful applications
of these structures are described in these papers, including applications to 
robust fuzzy extractors, secure multiparty computation, non-malleable codes, etc.
Various construction methods for AMD codes are also presented in these papers. 

We begin by providing some motivating examples as well as some
historical context from the point of view of
authentication codes. AMD codes can be 
considered as a variation of the classical {\it unconditionally secure authentication codes} \cite{Sim},
which we will refer to as  {\it $A$-codes} for short. An $A$-code has the form
$(\S, \TT, \K, \E)$  where $\S$ is a set of plaintext {\it sources}, $\TT$ is a set of 
{\it tags}, $\K$ is a set of {\it keys} and $\E$ is a set of {\it encoding functions}.
For each $K \in \K$, there is a (possibly randomized)
{encoding function} $E_K : \S \rightarrow \TT$. A secret key $K \in \K$ is chosen
randomly. Later a source $s \in \S$ is selected and the tag $t = E_K(s)$ is completed.
The tag $t$ is authenticated by verifying that $t = E_K(s)$; this can be done only with knowledge
of the key $K$. Having seen a valid pair $(s,t)$, 
an active adversary may create a bogus pair $(s',t')$ (where $s' \neq s$), hoping that it will be
accepted as authentic (this process is called {\it substitution}). 
The adversary is trying to maximize the {\it success probability}
of such an attack. One main objective is to design $A$-codes that will minimize the
success probability of the adversary.

\begin{example}
\label{ACexam}
Let $p$ be prime and define 
$\S = \TT = \zed_p$.
Define $\K = \zed_p \times \zed_p$.  
For every $K = (c,d)\in \K$,
define the function $E_K$ by the rule
$s \mapsto cs+d \bmod p$ for all $s \in \zed_p$.
(That is, the encoding functions consist of all linear functions 
from $\zed_p$ to $\zed_p$.) Any observed source-tag pair $(s,t)$
is valid under exactly $p$ of the $p^2$ keys. Then, any substitution $(s',t')$ ($s' \neq s$)
is valid under exactly $1$ of the $p$ ``possible'' keys. Therefore, the adversary's success
probability is $1/p$.
\end{example}

There are two types of AMD codes. The first type is a {\it weak AMD code}.
Here there is no key, so there is only one encoding function $E$.
Further, the tag is 
an element of a finite additive abelian group, say $\G$. The adversary
is required to commit to a specific substitution of the form
$g \mapsto g + \Delta$, where $\Delta \in \G \setminus \{0\}$ is fixed. 
Later, a source $s \in S$ is chosen randomly and encoded to $g = E(s)$.
Then $g$ is replaced by $g' = g + \Delta$. The adversary wins if
$g' = e(s')$ for some $s \neq s'$. Again, the objective in designing such a code
is to minimize the adversary's success probability.

\begin{example}
\label{weakexam}
Let $\S = \{1,2,3,4,5\}$ and let $\G = \zed_{21}$.
The encoding function $E$ is defined by
$E(1) = 3$,  $E(2) = 6$, $E(3) = 12$, $E(4) = 7$ and $E(5) = 14$.
It turns out that the adversary's success probability is $1/5$, independent of
his choice of $\Delta \neq 0$. This follows because $\{3,6,12,7,14\}$ is a
{\it difference set} in $\zed_{21}$ (for the definition of difference set, see Section
\ref{diffset.sec}).
\end{example}

The second type of AMD code is a {\it strong AMD code}.
It is basically the same as a weak AMD code, except that
 the adversary is given the source
(but not the encoded version of the source) before choosing $\Delta$.

\begin{example}
\label{strongexam}
This example is based on Example \ref{QR.ex}.
Let $\S = \{1,2,3,4\}$ and let $\G = \zed_{7}$.
The encoding function $E$ is defined by
$E(1) = 1$,  $E(2) = 2$, $E(3) = 4$ and $E(4) \in_R \{0,3,5,6\}$ 
(the notation ``$\in_R$'' denotes that the given encoding is to be chosen
uniformly at random from the given set).
If the source $s = 1,2$ or $3$, then the adversary succeeds with 
probability $1$ by choosing $\Delta$ such that $E(s) + \Delta = E(s')$ for some $s' \neq s$. 
However, if the source $s=4$, it can be verified that the adversary's success
probability is $1/2$. To see this, observe for any $\Delta \neq 0$ that $E(4) + \Delta
\in \{E(1),E(2),E(3)\}$ for precisely two of the four possible values of
$E(4)$. 
\end{example}

\subsection{Notation}

In this section, we present  relevant notation that we will use in the rest of the paper.

\begin{itemize}
\item
There is a set $\S$ of plaintext messages which is termed the {\it source space}, where $|\S| = m$.
There will be a probability distribution on $\S$, which is assumed
to be public. We will normally assume $\prob [s] = 1/m$ for all $s \in \S$, so we
have {\em equiprobable sources}.
\item The {\it encoded message space} (or more simply, {\it message space})  is a set $\G$, where $|\G| = n$ 
(note: $\G$ will usually be an additive abelian group with identity $0$).
\item For every source $s \in \S$, let $A(s) \subseteq \G$ denote the 
set of  {\it valid} encodings of $s$.
We require that $A(s) \cap A(s') = \emptyset$ if $s \neq s'$; 
this ensures that any message can be correctly decoded.
Denote $\A = \{ A(s) : s \in \S \}$.
\item Let $a_{s} = |A(s)|$ for  any $s \in \S$.
Define \[\G_0 = \bigcup _{s \in \S} A(s)\] and 
denote \[a = \sum_{s \in \S} a_s.\] If $a_{s}$ is constant, say $k$, 
then the code is {\it $k$-uniform}. In this case, $a = km$.
\item $E : \S \rightarrow G$ is a (possibly randomized) {\it encoding function}
 that maps a source $s \in \S$ to some $g \in A(s)$ 
 according to a certain probability distribution 
 defined on $A(s)$: 
 \[\prob [E(s) = g] = \prob[g \mid s].\]
 The encoding function $E$, as well as the probability
 distributions 
$\prob [E(s) = g]$, 
are assumed to be public.
Observe that, for equiprobable sources, the induced probability distribution on 
$\G_0$ is given by 
\[ \prob[g] = \frac{1}{m} \times \prob [E(s) = g]\]
for all $s \in \S$ and all $g \in A(s)$.
\item Formally, we can define the AMD code as a 4-tuple $(\S, \G, \A, E)$.
\item If  $\prob [E(s) = g] = 1/a_s$ 
for every  $s \in \S$ and every $g \in A(s)$, then the code has {\it equiprobable encoding}.
Such a code can be denoted as a 3-tuple $(\S, \G, \A)$.
In  a code with equiprobable sources and equiprobable encoding, we have \[ \prob[g] = \frac{1}{a_s m} \]
for all $s \in \S$ and all $g \in A(s)$.
\item A $k$-uniform code that has equiprobable sources and 
equiprobable encoding is said to be {\it $k$-regular}.
In  a $k$-regular code, we have \[ \prob[g] = \frac{1}{km} \]
for all $g \in \G_0$.
\item A $1$-regular code is said to be {\it deterministic}
because the source uniquely determines the encoding. 
In  a deterministic code with equiprobable sources, we have \[ \prob[g] = \frac{1}{m} \]
for all $g \in \G_0$.

\end{itemize}


\subsection{Formal Definitions of Weak and Strong AMD Codes}

We formally define the notion of {\it weak security} for an AMD code $(\S, \G, \A, E)$ by considering a 
certain game incorporating  an adversary. The adversary  has complete information about the AMD code 
that is being used. Based on this information, the adversary will adopt a \emph{strategy} $\sigma$
which he will use to choose a value $\Delta$ in the game described below. 
A strategy is allowed to be randomized.
\begin{definition}[Weak AMD code] \mbox{\quad}\vspace{.1in}\\
Suppose $(\S, \G, \A, E)$ is an AMD code.
\begin{enumerate}
\item The value $\Delta \in \G \setminus \{0\}$ is chosen according to the adversary's strategy.
\item The source $s \in \S$ is chosen uniformly at random by the encoder
(i.e., we have equiprobable sources).
\item The source is encoded into $g \in A(s)$ using the encoding function $E$.
\item The adversary  wins if and only if $g + \Delta \in A(s')$ for some $s' \neq s$.
\end{enumerate}
The \emph{success probability} of the strategy $\sigma$, denoted 
$\epsilon_{\sigma}$, is the probability that the adversary wins this game using the
specific strategy $\sigma$.

We will say that the code $(\S, \G, \A, E)$ is a \textbf{\textit{weak $(m,n,\he)$-AMD code}} where
$\he$ denotes the success probability of the adversary's optimal 
strategy.
That is, 
\[ \he = \max _{\sigma} \{ \epsilon_{\sigma} \} .\]
\end{definition}

We now turn to the stronger security model.
The following concept of {\it strong security} 
is also defined as a game involving an adversary.
In this model, the \emph{strategy} $\sigma$
used to choose $\Delta$ will depend on the source $s$.

\begin{definition}[Strong AMD code] \mbox{\quad}
\begin{enumerate}
\item The source $s \in \S$ is given to the adversary  
(here there is no probability distribution defined on $\S$).
\item The value $\Delta \in \G \setminus \{0\}$ is chosen according to the adversary's strategy.
\item The source is encoded into $g \in A(s)$ using the encoding function $E$.
\item The adversary  wins if and only if $g + \Delta \in A(s')$ for some $s' \neq s$.
\end{enumerate}
For a given source $s$ the \emph{success probability} of the strategy $\sigma$, denoted 
$\epsilon_{\sigma,s}$, is the probability that the adversary wins this game using the
specific strategy $\sigma$.

We will say that the code $(\S, \G, \A, E)$ is a \textbf{\textit{strong $(m,n,\he)$-AMD code}} where
$\he$ denotes the maximum success probability of any strategy over all sources $s$.
That is,
\[ \he = \max _{\sigma,s} \{ \epsilon_{\sigma,s} \} .\]
\end{definition}

As we mentioned earlier, the difference between a weak and strong AMD code
is that, in a weak code, the adversary 
chooses $\Delta$ before he sees $s$, while in a strong code, the adversary  
is given $s$ and then he chooses $\Delta$.

\subsection{Our Contributions}

In this paper, we study \emph{optimal} AMD codes, i.e., codes in which the 
adversary's success probability is as small as possible. We consider bounds
for both weak and strong AMD codes and investigate when these bounds can be
achieved. This involves several generalizations of difference families,
some of which have apparently not been studied previously. 

Connections between AMD codes and difference families have been observed
previously, e.g., in \cite{CFP}. The paper \cite{CFP} and other prior
work is mainly concerned with codes that are ``close to'' optimal and/or 
the construction of classes of codes that have \emph{asymptotically optimal} behaviour. 
This is of course desirable from the point of view of applications. 
In contrast, our focus is on mathematical characterizations of
codes where the relevant bounds are \emph{exactly} met with equality;
this is the sense in which we are using the term ``optimal''.

The rest of this paper is organized as follows.
In Section \ref{diffset.sec}, we define all the generalizations of difference
families that we will be using in the rest of the paper. We give some
examples and constructions as well as prove 
some nonexistence results. Section \ref{weak.sec} studies weak AMD codes.
Bounds are considered in Section \ref{weakbounds.sec}, where we introduce the
notion of \emph{R-optimal} and \emph{G-optimal} AMD codes; these bound arise
in the analysis of two different adversarial strategies. Conditions under which
these bounds can be met with equality are presented in Section \ref{weakoptimal.sec}.
Section \ref{strong.sec} provides an analogous treatment of strong AMD codes.
Finally, we conclude the paper in Section \ref{conclusion.sec}.

\section{Difference Families and Generalizations}
\label{diffset.sec}

In this section, we describe several variations of difference sets and
difference families. These concepts will be essential for constructions and
combinatorial characterizations of optimal (strong and weak) AMD codes.
Some of the definitions we give are new, and we prove some interesting
connections between various types of difference families that may be
of independent interest.

Let $\G$ be an abelian group.
For any two disjoint sets $A_1,A_2 \subseteq \G$, define 
\[ \D(A_1,A_2) =
\{ x-y : x \in A_1, y \in A_2\}.\]
Note that $\D(A_1,A_2)$ is a multiset.  Also, for any $A_1 \subseteq \G$, 
define 
\[ \D(A_1) = \{ x-y: x,y \in A_1, x \neq y\}.\]
$\D(A_1)$ is also a multiset.


Our first two definitions---difference sets and difference families---are standard.
There is a large literature on these combinatorial structures.

\begin{definition}[Difference Set]
Let $\G$ be an additive abelian group of order $n$.
An \textbf{\textit{$(n,m,\lambda)$-difference set}} (or \textbf{\textit{$(n,m,\lambda)$-DS}})
is a set  $A_1 \subseteq \G$, such that
the following multiset equation holds:
\[ \D(A_1) = \lambda (\G \setminus \{0\}).\]
\end{definition}
If an $(n,m,\lambda)$-DS exists, then $\lambda (n-1) = m (m-1)$.

\medskip

\noindent{\bf Remark:}  We can consider any set of size 1 to be a (trivial) difference set with
$\lambda = 0$.



\begin{definition}[Difference Family]
Let $\G$ be an additive abelian group of order $n$.
    An \textbf{\textit{$(n,m,k,\lambda)$-difference family}} (or \textbf{\textit{$(n,m,k,\lambda)$-DF}})
is a set of $m$ 
$k$-subsets of $\G$, say $A_1, \dots , A_m$, such that
the following multiset equation holds:
\[ \bigcup _{i}^{}\D(A_i)  = \lambda (\G \setminus \{0\}).\]
\end{definition}
If an $(n,m,k,\lambda)$-DF exists, then $\lambda (n-1) = m k (k-1)$.
Also, an $(n,m,\lambda)$-DS is an $(n,1,m,\lambda)$-DF.


The following definition is from  \cite{OKSS}.

\begin{definition}[External difference family]
Let $\G$ be an additive abelian group of order $n$.
An \textbf{\textit{$(n,m,k,\lambda)$-external difference family}} (or \textbf{\textit{$(n,m,k,\lambda)$-EDF}})
is a set of $m$ disjoint $k$-subsets of $\G$, say $A_1, \dots , A_m$, such that
the following multiset equation holds:
\[ \bigcup _{\{ i,j : i \neq j\} }^{}\D(A_i,A_j)  = \lambda (\G \setminus \{0\}).\]
\end{definition}
If an $(n,m,k,\lambda)$-EDF exists, then $n \geq mk$ and 
\begin{equation}
\label{EDF.eq}
\lambda (n-1) = k^2 m (m-1).
\end{equation}
Also, an $(n,m,1,\lambda)$-EDF is the same thing as an $(n,m,\lambda)$ difference set.

\medskip

There are several papers giving construction methods for  external difference families,
e.g., \cite{CD,CJX,FMY,FT,HW09,LF,Ton}.
Here is an example of one infinite class of external difference families.
due to Tonchev \cite{Ton}; it was later 
rediscovered in \cite{HW09}.

\begin{theorem}
\label{ton.thm}
\cite{Ton,HW09}
Suppose that $q = 2u\ell +1$ is a prime power, where $u$ and $\ell$ are odd.
Then there exists a $(q,u,\ell,(q-2\ell-1)/4)$-EDF in $\eff_q$.
\end{theorem}

\begin{proof}
Let $\alpha \in \eff_q$ be a primitive element. Let $C$ be the subgroup
of ${\eff_q}^*$ having order $u$ and index $2\ell$. The $\ell$ cosets
$\alpha^{2i}C$ ($0 \leq i \leq \ell-1$) form the EDF.
\end{proof}

\begin{example}
We give an example to illustrate Theorem \ref{ton.thm}.
Let $\G = (\zed_{19}, +)$. Then $\alpha = 2$ is a primitive element and
$C = \{ 1,7,11\} $ is the (unique) subgroup of order 3 in ${\zed_{19}}^*$.
A $(19,3,3,3)$-EDF is given by the three sets $\{ 1,7,11\}$, $\{ 4,9,6\}$ and $\{ 16,17,5\}$.
\end{example}

We refer to \cite[Table II]{FT} for a list of known external difference families.

\medskip

\noindent{\bf Remark:}  The related but more general concept of a {\it difference system of sets} was
defined much earlier, by Levenshtein, in \cite{Lev}. This is similar to the
definition of an external difference family, except that every difference 
$x-y$  ($x \in A_i, y \in A_j, i \neq j$) is required to occur 
{\it at least} $\lambda$ times. However, we note that
a {\it perfect, regular difference system of sets} is
equivalent to an external difference family.

\medskip

As we will discuss later, 
for the applications to AMD codes we will be 
considering, it is sufficient that every difference
occurs {\it at most} $\lambda$ times. This motivates the following definition.

\begin{definition}[Bounded external difference family]
Let $\G$ be an additive abelian group of order $n$.
A \textbf{\textit{$(n,m,k,\lambda)$-bounded external difference family}} 
(or \textbf{\textit{$(n,m,k,\lambda)$-BEDF}})
is a set of $m$ disjoint $k$-subsets of $\G$, say $A_1, \dots , A_m$, such that
the following condition holds for every $g \in \G \setminus \{0\}$:
\[ |\{ x-y: x-y = g, x \in A_i, y \in A_j,  i \neq j\}|  \leq \lambda .\]
\end{definition}
It is obvious that an $(n,m,k,\lambda)$-EDF is an $(n,m,k,\lambda)$-BEDF.

\begin{definition}[Strong external difference family]
Let $\G$ be an additive abelian group of order $n$.
An \textbf{\textit{$(n,m,k;\lambda)$-strong external difference family}}
(or \textbf{\textit{$(n,m,k;\lambda)$-SEDF}})
is a set of $m$ disjoint $k$-subsets of $\G$, say $A_1, \dots , A_m$, such that
the following multiset equation holds for every $i$, $1 \leq i \leq m$:
\begin{equation}
\label{SEDF.eq}  \bigcup _{\{ j : j \neq i\} }^{}\D(A_i,A_j)   = \lambda (\G \setminus \{0\}).
\end{equation}
\end{definition}

It is easy to see that a $(n,m,k,\lambda)$-SEDF
is an $(n,m,k,m\lambda)$-EDF. Therefore, from (\ref{EDF.eq}), if an $(n,m,k,\lambda)$-SEDF exists, then 
\begin{equation}
\label{SEDF-nc.eq}
\lambda (n-1) = k^2 (m-1).
\end{equation}

\begin{example}
\label{exam10}
Let $\G = (\zed_{k^2+1}, +)$, 
$A_1 = \{0, 1, \dots, k-1\}$ and $A_2 = \{k, 2k, \dots, k^2\}$.
This is a $(k^2+1,2;k;1)$-SEDF.
\end{example}

\begin{example}
\label{trivial-SEDF}
Let $\G = (\zed_{n}, +)$ and
$A_i = \{i\}$ for $0 \leq i \leq n-1$.
This is a $(n,n;1;1)$-SEDF.
\end{example}

\begin{theorem}
\label{SEDF-nonexistence}
There does not exist an $(n,m,k,1)$-SEDF with $m \geq 3$ and $k > 1$.
\end{theorem}

\begin{proof}
Suppose $A_1, \dots , A_m$ is an $(n,m,k,1)$-SEDF with $m \geq 3$ and $k > 1$.
From (\ref{SEDF.eq}), it follows that
\begin{equation}
\label{SEDF2.eq}  \bigcup _{\{ i,j: 1 \leq  i \leq m, 1 \leq j \leq m, i \neq j\} }^{}\D(A_i,A_j)   
= m (\G \setminus \{0\}).
\end{equation}
Then, from (\ref{SEDF.eq}) and (\ref{SEDF2.eq}), we have
\begin{equation}
\label{SEDF3.eq}  \bigcup _{\{ i,j: 2 \leq  i \leq m, 2 \leq j \leq m, i \neq j\} }^{}\D(A_i,A_j)   
= (m-2) (\G \setminus \{0\}).
\end{equation}
Suppose $x,y \in A_1$, $x \neq y$ (note that $k >1$ so we have two distinct elements in $A_1$).
Now, from (\ref{SEDF3.eq}), since $m > 2$, there exists $u \in A_i$, $v \in A_j$ such
that $i,j > 1$, $i \neq j$ and $u - v = x - y$.
Then $u-x = v-y$, which contradicts (\ref{SEDF.eq}).
\end{proof}

\begin{theorem}
\label{SEDF-existence}
There  exists an $(n,m,k,1)$-SEDF
if and only if $m=2$ and $n = k^2+1$, or
$k = 1$ and $m = n$.
\end{theorem}

\begin{proof}
From Theorem  \ref{SEDF-nonexistence}, we only need to consider the cases
$m = 2$ and $k=1$. If $m=2$, then from (\ref{SEDF-nc.eq}), we must have $n = k^2+1$, and the relevant SEDF
exists from Example \ref{exam10}. If $k=1$,  then from (\ref{SEDF-nc.eq}) we must have $m=n$, and the relevant 
SEDF exists from Example \ref{trivial-SEDF}.
\end{proof}

Next, we consider generalizations of external difference families and
strong external difference families in which the subsets $A_1, \dots , A_m$
are allowed to be of possibly different sizes.

\begin{definition}[Generalized external difference family]
Let $\G$ be an additive abelian group of order $n$.
An \textbf{\textit{$(n,m;k_1, \dots , k_m;\lambda)$-generalized external difference family}} 
(or \textbf{\textit{$(n,m;k_1, \dots , k_m;\lambda)$-GEDF}})
is a set of $m$ disjoint subsets of $\G$, say $A_1, \dots , A_m$, such that
$|A_i| = k_i$ for $1 \leq i \leq m$ and 
the following multiset equation holds:
\[ \bigcup _{\{ i,j : i \neq j\} }^{}\D(A_i,A_j)   = \lambda (\G \setminus \{0\}).\]
\end{definition}
Clearly, an $(n,m,k,\lambda)$-EDF is an $(n,m;k, \dots , k;\lambda)$-GEDF.

\begin{example}
\label{exam13}
Let $\G = (\zed_{13}, +)$, 
$A_1 = \{0,1\}$ and $A_2 = \{2,4,6\}$.
This is a $(13,2;2,3;1)$-GEDF.
\end{example}

\begin{example}
Let $\G = (\zed_{11}, +)$, 
$A_1 = \{0\}$, $A_2 = \{1\}$, and $A_3 = \{3,5\}$.
This is a $(11,3;1,1,2;1)$-GEDF.
\end{example}

\medskip

\noindent{\bf Remark:} A generalized external difference family is also known as 
a {\it perfect difference system of sets}.

\begin{definition}[Generalized strong external difference family]
Let $\G$ be an additive abelian group of order $n$.
An \textbf{\textit{$(n,m;k_1, \dots , k_m;\lambda_1, \dots , \lambda_m)$-generalized strong external difference family}} 
(or \textbf{\textit{$(n,m;k_1, \dots , k_m;\lambda_1, \dots , \lambda_m)$-GSEDF}})
is a set of $m$ disjoint subsets of $\G$, say $A_1, \dots , A_m$, such that
$|A_i| = k_i$ for $1 \leq i \leq m$ and 
the following multiset equation holds for every $i$, $1 \leq i \leq m$:
\[ \bigcup _{\{ j : j \neq i\} }^{}\D(A_i,A_j)   = \lambda_i (\G \setminus \{0\}).\]
\end{definition}
It is obvious that an $(n,m,k,\lambda)$-SEDF is an 
$(n,m;k, \dots , k;\lambda, \dots , \lambda)$-GSEDF.

\begin{example}
Let $\G = (\zed_{n}, +)$, 
$A_1 = \{0\}$ and $A_2 = \{1, 2, \dots , n-1\}$.
This is a $(n,2;1,n-1;1,1)$-GSEDF.
\end{example}

\begin{example}
\label{QR.ex}
Let $\G = (\zed_{7}, +)$, 
$A_1 = \{1\}$, $A_2 = \{2\}$, $A_3 = \{4\}$, and $A_4 = \{0,3,5,6\}$.
This is a $(7,4;1,1,1,4;1,1,1,2)$-GSEDF.
\end{example}

A $(n,m;k_1, \dots , k_m;\lambda_1, \dots , \lambda_m)$-GSEDF is {\it maximal} if 
$\sum k_i = n$.  Here is a nice characterization of maximal GSEDF.

\begin{theorem}
\label{DS-GSEDF.thm}
Suppose $A_1, \dots , A_m$ is a partition of $\G$ (where $|\G| = n$) with $|A_i| = k_i$ for $1 \leq i \leq m$.
Then $A_1, \dots , A_m$ is a (maximal) $(n,m;k_1, \dots , k_m;\lambda_1, \dots , \lambda_m)$-GSEDF
if and only if $A_i$ is an $(n,k_i,k_i-\lambda_i)$-DS in $\G$, for $1 \leq i \leq m$.
\end{theorem}

\begin{proof}
Fix a value $i$, $1 \leq i \leq m$.  It is clear that 
 \begin{eqnarray*}
 \bigcup _{\{ j : j \neq i\} }^{}\D(A_i,A_j) &=& \D(A_i, \G \setminus A_i)\\
 &=& \bigcup _{x \in A_i} \D(x,\G \setminus A_i)\\
 &=& \bigcup _{x \in A_i} \left( \D(x,\G \setminus \{x\} ) \setminus \D(x,A_i \setminus \{x\})  \right) \\
 &=& \left( \bigcup _{x \in A_i} \D(x,\G \setminus \{x\} ) \right) \setminus 
 \left( \bigcup _{x \in A_i}\D(x,A_i \setminus \{x\})  \right) \\
 &=& \left( \bigcup _{x \in A_i} \G \setminus \{0\} ) \right) \setminus 
 \D(A_i)  \\
 &=& (k_i ( \G \setminus \{0\} ) ) \setminus 
 \D(A_i) , 
 \end{eqnarray*}
where all operations are multiset operations.
Therefore, 
\[ \bigcup _{\{ j : j \neq i\} }^{}\D(A_i,A_j) =  \lambda_i( \G \setminus \{0\} )\]
if and only if 
\[ \D(A_i) = (k_i-\lambda_i) ( \G \setminus \{0\} ) .\]
\end{proof}

\begin{theorem}
Suppose there exists an $(n,m;k_1, \dots , k_m;\lambda_1, \dots , \lambda_m)$-GSEDF
where $k_i = 1$. Then $\lambda_i = 1$ and $\sum_{i=1}^{m} k_i = n$
(i.e., the GSEDF is maximal).
\end{theorem}

\begin{proof}
We have $k_i (a - k_i) = a-1 = \lambda_i (n-1)$, where $a = \sum_{i=1}^{m} k_i$.
Since $a \leq n$ and $\lambda_i \geq 1$, it must be the case that
$a=n$ and $\lambda_i = 1$.
\end{proof}

\begin{definition}[Bounded generalized strong external difference family]
Let $\G$ be an additive abelian group of order $n$.
An \textbf{\textit{$(n,m;k_1, \dots , k_m;\lambda_1, \dots , \lambda_m)$-bounded generalized strong external difference family}} 
(or \textbf{\textit{$(n,m;k_1, \dots , k_m;\lambda_1, \dots , \lambda_m)$-BGSEDF}})
is a set of $m$ disjoint subsets of $\G$, say $A_1, \dots , A_m$, such that
$|A_i| = k_i$ for $1 \leq i \leq m$ and 
the following multiset equation holds for every $j$, $1 \leq j \leq m$, 
and for every $g \in \G \setminus \{0\}$:
\[ |\{ x-y: x-y = g, x \in A_i, y \in A_j,  i \neq j\}|  \leq \lambda_j .\]
\end{definition}

\noindent{\bf Remark:}  A BGSEDF is equivalent to the notion of \emph{differential structure},
as defined, e.g.,  in \cite{CFP}.

\begin{definition}[Partitioned external difference family]
Let $\G$ be an additive abelian group of order $n$.
An \textbf{\textit{$(n,m;c_1, \dots , c_{\ell};k_1, \dots , k_{\ell};\lambda_1, \dots , \lambda_{\ell})$-partitioned external difference family}}
(or \textbf{\textit{$(n,m;c_1, \dots , c_{\ell};k_1, \dots , k_{\ell};\lambda_1, \dots , \lambda_{\ell})$-PEDF}})
is a set of $m= \sum_i c_i$ disjoint subsets of $\G$, say $A_1, \dots , A_m$, such that
there are $c_h$ subsets of size $k_h$, for $1 \leq h \leq \ell$, and 
the following multiset equation holds for every $h$, $1 \leq h \leq \ell$:
\[ \bigcup _{\{ i : |A_i| = c_h\} }^{} \bigcup _{\{ j : j \neq i\} }^{}\D(A_i,A_j)   = \lambda_i (\G \setminus \{0\}).\]
\end{definition}

We note the following:
\begin{itemize}
\item an $(n,m;k_1, \dots , k_m;\lambda_1, \dots , \lambda_m)$-GSEDF
is an    $(n,m;1, \dots , 1;k_1, \dots , k_{m};\lambda_1, \dots , \lambda_{m})$-PEDF
\item an $(n,m,k,\lambda)$-EDF is an 
         $(n,m;m;k;\lambda)$-PEDF
\item an $(n,m;c_1, \dots , c_{\ell};k_1, \dots , k_{\ell};\lambda_1, \dots , \lambda_{\ell})$-PEDF
is an    $(n,m;{k_1}^{c_1}, \dots , {k_{\ell}}^{c_{\ell}};\lambda)$-GEDF in which
\[ \lambda = \sum _{i=1}^{\ell} \lambda_i,\]
where the notation ${k_i}^{c_i}$ denotes $c_i$ occurrences of $k_i$, for $1 \leq h \leq \ell$.
\end{itemize}
Here is an example of a PEDF that is not an EDF or GSEDF.

\begin{example}
\label{GSSEDF.ex}
Let $\G = (\zed_{13}, +)$, 
$A_1 = \{0,1,4\}$, $A_2 = \{3,5,10\}$, $A_3 =  \{2,6,7,9\}$, $A_4 =  \{8\}$, 
$A_5 =  \{11\}$,  $A_6 =  \{12\} $.
It can be verified that $A_1, \dots , A_6$ is a $(13,6;2,1,3;3,4,1;5,3,3)$-PEDF. 
To see that it is not a GSEDF, we first compute the 
occurrence of differences
from $A_1$ to the union of the other $A_i$'s:
\[ 
\begin{array}{crrrrrrrrrrrr}
\mathrm{difference} & 1 & 2 & 3 & 4 & 5 & 6 & 7 & 8 & 9 & 10 & 11 & 12\\ \hline
\mathrm{frequency} & 2 & 3 & 2 & 2 & 3 & 3 & 3 & 3 & 2 & 2 & 3 & 2
\end{array}
\]
Then we compute the 
occurrence of differences
from $A_2$ to the union of the other $A_i$'s:
\[ 
\begin{array}{crrrrrrrrrrrr}
\mathrm{difference} & 1 & 2 & 3 & 4 & 5 & 6 & 7 & 8 & 9 & 10 & 11 & 12\\ \hline
\mathrm{frequency} & 3 & 2 & 3 & 3 & 2 & 2 & 2 & 2 & 3 & 3 & 2 & 3
\end{array}
\]
These two lists of occurrences of differences are not uniform, so we do not have a GSEDF.
However, each difference occurs a total of five times in the two lists.
\end{example}

\begin{theorem}
Suppose $A_1, \dots , A_m$ is a partition of $\G$ (where $|\G| = n$) such that there 
are $c_h$ subsets of size $k_h$ for $1 \leq h \leq \ell$.
Then $A_1, \dots , A_m$ is a (maximal) 
$(n,m;c_1, \dots , c_{\ell};k_1, \dots , k_{\ell};\lambda_1, \dots , \lambda_{\ell})$-PEDF
if and only if the subsets of cardinally $k_h$ form an  $(n,k_h,c_hk_h-\lambda_h)$-DF in $\G$, 
for $1 \leq h \leq \ell$.
\end{theorem}

\begin{proof}
We omit the proof, which is similar to the proof of Theorem \ref{DS-GSEDF.thm}.
\end{proof}

\begin{example}
\label{GSSEDF2.ex}
Let's look again at the PEDF in Example \ref{GSSEDF.ex}.
Here the two sets of size 3 form a $(13,2,3,1)$-DF;
the set of size 4 is a $(13,1,4,1)$-DF; and the three sets of size 1 form
a $(13,3,1,0)$-DF.
\end{example}

In Figure \ref{diff-fig}, we indicate the relationship between the
various types of difference families we have defined. 
If we designate $X \rightarrow Y$, this indicates that
any example of ``$X$''  automatically satisfies the properties of ``$Y$''.

\begin{figure}[tb]
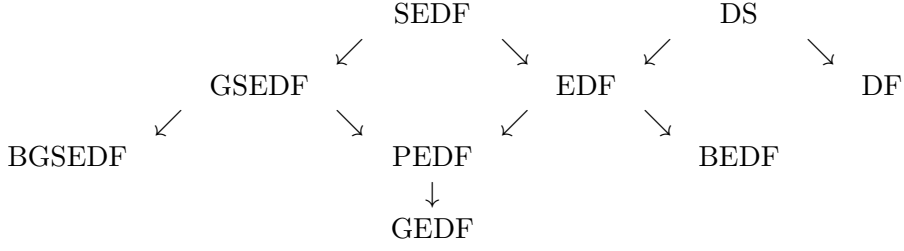

\[
\begin{array}{ccccccccccc}
& &  & & \mathrm{SEDF} & & & &  \mathrm{DS}\\
& &  & \swarrow &  & \searrow & & \swarrow &  & \searrow \\
& & \mathrm{GSEDF} & & & & \mathrm{EDF} & & & & \mathrm{DF}\\ 
& \swarrow &  & \searrow &  & \swarrow & & \searrow &  &  \\
\mathrm{BGSEDF} & &  & & \mathrm{PEDF} & & & &  \mathrm{BEDF}\\
& &  & & \downarrow \\
& &  & & \mathrm{GEDF}
\end{array}
\]
\caption{Relationships between 
various types of difference families}
\label{diff-fig}
\end{figure}

\section{Weak AMD Codes}
\label{weak.sec}

Our goal is to prove lower bounds on the adversary's optimal success probability, $\he$. 
Note that a {\it lower} bound on $\he$ states that
there exists an adversary  who wins the relevant game with {\it at least} some specified
probability. Then we  construct codes 
that meet these lower bounds, i.e., codes in which the adversary cannot succeed
with higher probability.  Whenever possible, we will prove bounds without assuming that
the code is uniform or has equiprobable encoding (we do assume equiprobable sources, however).

\subsection{Bounds for Weak AMD Codes}
\label{weakbounds.sec}

%

\begin{theorem}
\label{weak-bound}
In any weak $(m,n,\he)$-AMD code, it holds that \[\he \geq \frac{a(m-1)}{m(n-1)}.\]
\end{theorem}

\begin{proof}
Suppose the adversary   chooses the value $\Delta \in \G \setminus \{0\}$ 
uniformly at random. 
For any given $g \in A(s)$ and for a randomly chosen 
$\Delta$, the probability that
the adversary  wins is $(a - a_s) / (n-1)$.
The success probability $\epsilon_{\mathsf{rand}}$ of this random strategy $\mathsf{rand}$ is
\begin{eqnarray*}
\epsilon_{\mathsf{rand}} &=& \sum_s \prob[s] \sum_{g \in A(s)} \left( \prob [E(s) = g] \times \frac{a - a_s}{n-1}\right)\\
& = & \sum_s \left( \prob[s] \times \frac{a - a_s}{n-1} \right)\\
& = & \frac{a}{n-1} - \sum_s \frac{a_s}{m(n-1)} \quad \mbox{(because the sources are equiprobable)} \\
& = & \frac{a}{n-1} - \frac{a}{m(n-1)}\\
& = & \frac{a(m-1)}{m(n-1)}.
\end{eqnarray*}
\end{proof}

\begin{corollary}
\label{weak-bound-cor-uniform}
In any $k$-uniform weak $(m,n,\he)$-AMD code, it holds that \[\he \geq \frac{k(m-1)}{n-1}.\]
\end{corollary}

\begin{proof}
Note that $a = km$ in a $k$-uniform  code and apply Theorem \ref{weak-bound}.
\end{proof}

\begin{definition}
\label{weak-R-optimal}
We will define a weak AMD code that meets the bound of Theorem \ref{weak-bound} 
(or Corollary \ref{weak-bound-cor-uniform}, in the
case that the code is $k$-uniform) with equality
to be \emph{R-optimal}. Here, ``R'' is used to indicate that $\mathsf{rand}$ is an optimal strategy. 
\end{definition}

\begin{corollary}\cite[Theorem 2.2]{CFP}
\label{weak-bound-cor}
In any weak $(m,n,\he)$-AMD code, it holds that \[\he \geq \frac{m-1}{n-1}.\]
\end{corollary}

\begin{proof}
Note that $a \geq m$ and apply Theorem \ref{weak-bound}.
\end{proof}

\noindent{\bf Remark:}  
The bound of Corollary \ref{weak-bound-cor} is met with equality only if 
the code is deterministic.

\medskip

Here is a new bound for weak AMD codes, that arises from a different adversarial strategy.

\begin{theorem}
\label{weak-bound-2}
In any weak $(m,n,\he)$-AMD code, it holds that \[\he \geq \frac{1}{a}.\]
\end{theorem}

\begin{proof}
We consider the following strategy $\mathsf{guess}$ for the adversary:
\begin{enumerate}
\item Find the encoding $\hat{g} \in \A$ that occurs with the highest probability.
Observe that $\prob[\hat{g}] \geq 1/a$.
\item Pick a $\Delta$ that will work for the particular encoding $\hat{g}$.
\end{enumerate}
Clearly, the success probability $\epsilon_{\mathsf{guess}}$ of the strategy $\mathsf{guess}$ is 
equal to $\prob[\hat{g}] \geq 1/a$.
\end{proof}

\begin{definition}
\label{weak-G-optimal}
We will define a weak AMD code that meets the bound of Theorem \ref{weak-bound-2} 
with equality
to be \emph{G-optimal}. Here, ``G'' is used to indicate that $\mathsf{guess}$ is an optimal strategy.
\end{definition}

\begin{theorem}
\label{weak-bound-new}
In any weak $(m,n,\he)$-AMD code, it holds that \[\he^2 \geq \frac{m-1}{m(n-1)}.\]
\end{theorem}

\begin{proof}
Multiply the bounds proven in Theorems \ref{weak-bound} and \ref{weak-bound-2}.
\end{proof}

A code that meets the bound of Theorem \ref{weak-bound-new} with equality is simultaneously 
R-optimal and G-optimal.


\subsection{Optimal Weak AMD Codes}
\label{weakoptimal.sec}

In this section, we consider weak AMD codes that are R-optimal and/or G-optimal.
Recall that a weak AMD code is R-optimal if $\he  = a(m-1)/(m(n-1))$ and
it is G-optimal if   $\he  = 1/a$.

\subsubsection{R-Optimal Weak AMD Codes}

First, we consider R-optimality.
Consider the strategy $g \mapsto g + \Delta$, where $\Delta \neq 0$, and 
let $\epsilon_{\Delta}$ denote the success probability of this strategy.
Clearly, we have
\begin{equation}
\label{goodhat.eq} 
\he = 
\max \{ \epsilon_{\Delta} : \Delta \neq 0\}.
\end{equation}

For any $\Delta \neq 0$, define 
\begin{equation}
\label{good.eq} 
\good(\Delta) = \{ g \in \G_0: g \in A(s) \text{ and } g + \Delta \in A(s'), \text{where } s' \neq s\} .
\end{equation}
$\good(\Delta)$ denotes the set of encodings $g$ under which a substitution
$g \mapsto g + \Delta$ will result in the adversary winning the game.

\begin{lemma}
For any $\Delta \neq 0$,  it holds
that
\begin{equation}
\label{goodh2.eq} 
\epsilon_{\Delta} = \sum_{g \in \good(\Delta)} \prob[g].
\end{equation}
\end{lemma}

\begin{proof}
It is clear that 
\begin{eqnarray*} \epsilon_{\Delta} &=& \prob [ g  \in \good(\Delta)]\\
& = & \sum_{g \in \good(\Delta)} \prob[g].
\end{eqnarray*}
\end{proof}

\begin{theorem}
\label{weak-probs}
A weak AMD code is R-optimal 
if and only if $\epsilon_{\Delta} = a(m-1)/(m(n-1))$ 
for all $\Delta \neq 0$.
\end{theorem}

\begin{proof}
Suppose we have an R-optimal weak AMD code.
It is not hard to compute
\begin{eqnarray*}
\sum _{\Delta \neq 0} \epsilon_{\Delta} &=& \sum _{\Delta \neq 0} \sum_{g \in \good(\Delta)} \prob[g]\\
&=& \sum_{g \in \G_0} \prob[g] \times | \{\Delta :g \in \good(\Delta)  \}|  \\
&=& \sum _{s \in \S} \sum_{g \in A(s)} \prob[s]\, \prob[E(s) = g] \times | \{\Delta :g \in \good(\Delta)  \}|  \\
&=& \sum _{s \in \S} \prob[s] \sum_{g \in A(s)} \prob[E(s) = g]  (a-a_s)    \\
&=& \sum _{s \in \S} \prob[s]  (a-a_s)    \\
&=& \sum _{s \in \S} \frac{1}{m}  (a-a_s)   \\
&=& \frac{a(m-1)}{m}.
\end{eqnarray*}
Therefore the average of the quantities $\epsilon_{\Delta}$ ($\Delta \neq 0$) is equal to 
$a(m-1)/(m(n-1))$.
In order to have $\he = a(m-1)/(m(n-1))$, it must be the case that
$\epsilon_{\Delta} = a(m-1)/(m(n-1))$ for all $\Delta \neq 0$.
\end{proof}

We next present a method of constructing R-optimal weak AMD codes.

\begin{theorem}
\label{GSEDF-construction}
Suppose there is an $(n,m;k_1, \dots , k_m;\lambda_1, \dots , \lambda_m)$%
-GSEDF.   
Then there is an (R-optimal) weak  $(m,n,a(m-1)/(m(n-1)))$-AMD code,
where $a = \sum_{i=1}^{m} k_i$.
\end{theorem}

\begin{proof}
Suppose the GSEDF is given by $A_1, \dots , A_m$.
Let $a = \sum_{i=1}^{m} k_i$.
Observe that 
\begin{equation}
\label{GESDF.eq}
k_i(a-k_i) = \lambda_i(n-1)
\end{equation} for $1 \leq i \leq m$.
Let $\S = \{s_1, \dots , s_m\}$ be a set of $m$ sources. For $1 \leq i \leq m$, 
define $A(s_i) = A_i$ and suppose
the encoding function $E(s_i)$ is equiprobable.
We show that $\epsilon_{\Delta} = a(m-1)/(m(n-1))$ for all $\Delta \neq 0$.
We have 
\begin{eqnarray*}
\epsilon_{\Delta} &=& 
\sum_{g \in \good(\Delta)} \prob[g]\\
 & = & \sum_{i=1}^{m} \frac{1}{m} \times \frac{\lambda_i}{k_i}\\
& = & \frac{1}{m} \sum_{i=1}^{m}   \frac{a - k_i}{n-1} \quad \text{from (\ref{GESDF.eq})}\\
& = & \frac{1}{m(n-1)} \sum_{i=1}^{m}   (a - k_i)\\
& = & \frac{a(m-1)}{m(n-1)}.
\end{eqnarray*}
\end{proof}

In fact, we can obtain R-optimal weak AMD codes from a weaker type of
difference family, namely, a PEDF. 

\begin{theorem}
\label{PEDF-construction}
Suppose there is an $(n,m;c_1, \dots , c_{\ell};k_1, \dots , k_{\ell};\lambda_1, \dots , \lambda_{\ell})$-PEDF.   
Then there is an (R-optimal) weak  $(m,n,a(m-1)/(m(n-1)))$-AMD code,
where $a = \sum_{h=1}^{\ell} c_hk_h$.
\end{theorem}

\begin{proof}
We omit the proof, which is similar to the proof of Theorem  \ref{GSEDF-construction}.
\end{proof}

It is interesting to note that the we do not necessarily obtain an R-optimal AMD code if we start from an
arbitrary generalized external difference family. As an example, suppose we construct
an AMD code with equiprobable encoding for two sources using the GEDF presented in Example \ref{exam13}.
Here it is easy to compute
\[ \epsilon_{1} = \frac{1}{4} >  \frac{a(m-1)}{m(n-1)} = \frac{5\times 1}{2 \times 12} = \frac{5}{24},\] 
so this code is not R-optimal

It is an open problem to characterize R-optimal (weak) AMD codes.
The following example illustrates that the converse of  Theorem \ref{PEDF-construction} is not true in general.
That, is we can construct R-optimal codes that do not come from PEDFs.

\begin{example}
Let $\S = \{1,2,3,4\}$ and let $\G = \zed_{10}$.
The encoding function $E$ is defined by
$E(1) = 0$,  $E(2) = 5$, $E(3) \in_R \{1,9\}$ and $E(4) \in_R \{2,3\}$.

Suppose the adversary chooses $\Delta = 5$; then the adversary wins if
$s \in \{1,2\}$, which occurs with probability $1/2$. Suppose the adversary chooses $\Delta = 1$;
then the adversary succeeds if $s \in \{1,3\}$, which occurs with probability $1/2$.
Suppose the the adversary chooses $\Delta = 2$;
then the adversary succeeds if $s=1$, if $s=3$ and $E(s) = 1$, or
if $s=4$ and $E(s) = 3$. The success probability here is
\[ \frac{1}{4} + \frac{1}{4} \times \frac{1}{2} + \frac{1}{4} \times \frac{1}{2} =  \frac{1}{2}.\]
The remaining choices for $\Delta$ can be checked in a similar way. We obtain a code with
success probability $1/2$. 
Since $m = 4$, $n = 10$ and $a = 6$,
we have $a(m-1)/(m(n-1)) = 18/36 = 1/2$, so the code is R-optimal. 
However, the sets $\{0\},\{5\},\{1,9\},\{2,8\}$
do not form a PEDF. 
\end{example}

We can give a tight characterization of \emph{$k$-regular} R-optimal weak AMD  codes, however, as follows.
\begin{theorem}
\label{weak-equiv}
An (R-optimal) $k$-regular weak $(m,n,k(m-1)/(n-1))$-AMD code is equivalent to an $(n,m,k,\lambda)$-EDF.
\end{theorem}

\begin{proof}
Suppose $A_1, \dots , A_m$ is an $(n,m,k,\lambda)$-EDF.
Let $\S = \{s_1, \dots , s_m\}$ be a set of $m$ sources. For $1 \leq i \leq m$, suppose
the encoding function $E(s_i)$ is equiprobable.
The resulting weak AMD code is $k$-regular.
Choose any $\Delta \in \G$, $\Delta \neq 0$. 
The  strategy $g \mapsto g + \Delta$ succeeds with 
probability $\epsilon_{\Delta} = \lambda /(km) = k(m-1)/(n-1)$.
(In fact, this follows from Theorem \ref{PEDF-construction}.)

Conversely, suppose we have an R-optimal $k$-regular weak AMD code. 
Then it must be the case that $\epsilon_{\Delta} = k(m-1)/(n-1)$ for all $\Delta \neq 0$.
Using the fact that the code is a $k$-regular AMD, 
we have
\[ \frac{k(m-1)}{n-1} = \epsilon_{\Delta} = \prob [ E(s) \in \good(\Delta)] = \frac{|\good(\Delta)|}{km} .\]
Therefore,  
\[ |\good(\Delta)| = \frac{k^2m(m-1)}{n-1}.\]
It then follows that $\{ A(s) : s \in \S\}$ is an 
$(n,m,k,\lambda)$-EDF, where $\lambda = k^2m(m-1)/(n-1)$.
\end{proof}

In Figure \ref{diff-AMD-fig} we indicate the types of difference
families that yield R-optimal weak AMD codes. This summarizes the results proven in this section.

\begin{figure}[tb]
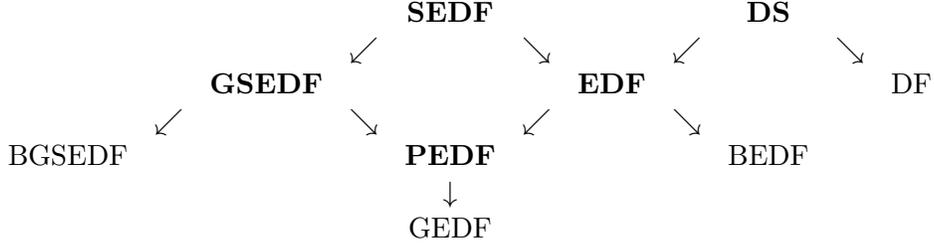

\[
\begin{array}{ccccccccccc}
& &  & & \mathbf{SEDF} & & & &  \mathbf{DS}\\
& &  & \swarrow &  & \searrow & & \swarrow &  & \searrow \\
& & \mathbf{GSEDF} & & & & \mathbf{EDF} & & & & \mathrm{DF}\\ 
& \swarrow &  & \searrow &  & \swarrow & & \searrow &  &  \\
\mathrm{BGSEDF} & &  & & \mathbf{PEDF} & & & &  \mathrm{BEDF}\\
& &  & & \downarrow \\
& &  & & \mathrm{GEDF}
\end{array}
\]
\caption{Difference
families that yield R-optimal weak AMD codes
(indicated in boldface type)}
\label{diff-AMD-fig}
\end{figure}

\subsubsection{G-Optimal Weak AMD Codes}

Now we turn to G-optimality. We have the following characterization of G-optimal
weak AMD codes.

\begin{theorem}
\label{weak-equiv-2}
A (G-optimal) weak $\left(m,n,\frac{1}{a}\right)$-AMD code is equivalent to an $(n,m,k,1)$-BEDF,
where $a = km$.
\end{theorem}

\begin{proof}
Suppose $A_1, \dots , A_m$ is an $(n,m,k,1)$-BEDF.
Let $\S = \{s_1, \dots , s_m\}$ be a set of $m$ sources. For $1 \leq i \leq m$, define
an encoding function $E(s_i)$ which chooses an element of $A_i$ uniformly at random.
Choose any $\Delta \in \G$, $\Delta \neq 0$. 
The strategy $g \mapsto g + \Delta$ succeeds with 
probability $\epsilon_{\Delta} \leq 1 /(km) = 1/a$, since there is at most 
one occurrence of the difference $\Delta$ in the BEDF. Further, if
$\Delta \in \D(A_j,A_i)$ where $i \neq j$, then the strategy $g \mapsto g + \Delta$ succeeds with 
probability $1/a$.

Conversely, suppose we have a G-optimal weak AMD code. 
From the proof of Theorem \ref{weak-bound-2}, we see that
all encodings must occur with the same probability, $1/a$. Since the sources
are equiprobable, this happens only if the 
code is $k$-regular with $k = a/m$. Now
we claim that $\{ A(s) : s \in \S\}$ is an 
$(n,m,k,1)$-BEDF. This is easy to see, because if some difference occurred more than once,
it would immediately follow that  $\epsilon \geq 2/a$.
\end{proof}


Now we characterize $k$-regular weak AMD codes that are simultaneously R-optimal and G-optimal.

\begin{theorem}
\label{weak-equiv-new}
A $k$-regular weak 
AMD code that is simultaneously R-optimal and G-optimal 
is equivalent to an $(n,m,k,1)$-EDF.
\end{theorem}

\begin{proof}
From Theorem \ref{weak-bound-new}, the code has success probability $\sqrt{\frac{m-1}{m(n-1)}}$.
In order for this to occur, the bounds of Corollary \ref{weak-bound-cor-uniform} and \ref{weak-bound-2} both must hold with equality.
Therefore the AMD code is simultaneously an $(n,m,k,\lambda)$-EDF (from Theorem \ref{weak-equiv}) and an $(n,m,k,1)$-BEDF
(from Theorem \ref{weak-equiv-2}).
Hence, it is an $(n,m,k,1)$-EDF.
\end{proof}

In Figure \ref{diff-AMD-fig-2} we indicate the types of difference
families that yield G-optimal weak AMD codes. Note that the relevant difference
families are assumed to have $\lambda=1$ in this figure.

\begin{figure}[tb]
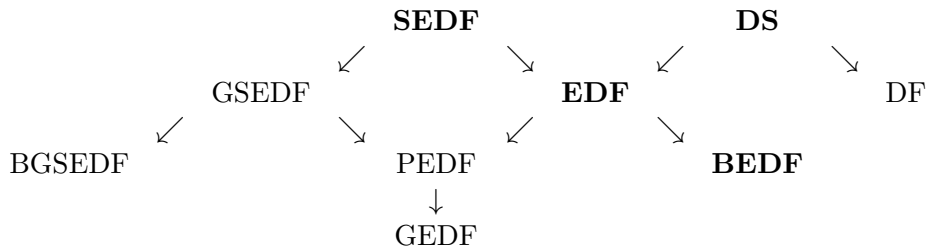

\[
\begin{array}{ccccccccccc}
& &  & & \mathbf{SEDF} & & & &  \mathbf{DS}\\
& &  & \swarrow &  & \searrow & & \swarrow &  & \searrow \\
& & \mathrm{GSEDF} & & & & \mathbf{EDF} & & & & \mathrm{DF}\\ 
& \swarrow &  & \searrow &  & \swarrow & & \searrow &  &  \\
\mathrm{BGSEDF} & &  & & \mathrm{PEDF} & & & &  \mathbf{BEDF}\\
& &  & & \downarrow \\
& &  & & \mathrm{GEDF}
\end{array}
\]
\caption{Difference
families with $\lambda = 1$ that yield G-optimal weak AMD codes
(indicated in boldface type)}
\label{diff-AMD-fig-2}
\end{figure}

\section{Strong AMD Codes}
\label{strong.sec}

We begin by focussing on the success probability of the adversary when the source 
is fixed to be $s$. 
%
%
Let $\he_s$ be the success probability of the optimal strategy for the given source $s$.



\begin{theorem}
\label{strong-bound-R}
In any strong AMD code, it holds that \[ \he_s \geq \frac{a-a_{s}}{n-1}\]
for any source $s \in \S$. 
\end{theorem}

\begin{proof}
As in the proof of Theorem \ref{weak-bound}, we consider a random strategy, 
i.e., $\Delta \neq 0$ is chosen uniformly at random. 
Given that the source is $s$,  
it is easy to see that 
the success probability of this strategy will be 
\[ \frac{a-a_s}{n-1} .\]
\end{proof}

\begin{definition}
\label{strong-R-optimal}
We will define a strong AMD code that meets the bound of Theorem \ref{strong-bound-R} 
with equality for every possible source $s$
to be \emph{R-optimal}. Again, ``R'' is used to indicate that choosing $\Delta \neq 0$ uniformly
at random is an optimal strategy. 
\end{definition}

\begin{corollary}
\label{strong-bound-1}
In any strong $(m,n,\he)$-AMD code, it holds that \[\he \geq \frac{a-a_{s'}}{n-1},\]
where $a_{s'} = \min \{ a_s :s \in \S\}$. 
\end{corollary}
\begin{proof} 
The quantity $(a-a_{s})/(n-1)$ is maximized when $a_{s}$ is minimized.
\end{proof}

If the code is $k$-uniform, then the previous bound takes a simpler form.

\begin{corollary}
\label{strong-bound-2}
In any $k$-uniform strong $(m,n,\he)$-AMD code, it holds that \[\he \geq \frac{k(m-1)}{n-1}.\]
\end{corollary}
\begin{proof} 
Here $a_s = k$ for all $s$ and $a = km$. Apply Corollary \ref{strong-bound-1}.
\end{proof}

\begin{theorem}
\label{strong-bound-G}
In any strong AMD code, 
it holds that $\he_s  \geq  1/a_s,$
for any source $s \in \S$. 
\end{theorem}

\begin{proof}
Given any source $s$, the 
adversary   can try to guess the encoded message $E(s)$ that is output. The adversary  will maximize 
his probability of success by choosing $g$ such that $\prob[g \mid s]$ is maximized.
Note that there exists a $g$ such that $\prob[g \mid s] \geq 1/a_s$.
Then the adversary  can choose $\Delta$ such that $g + \Delta \in  \G_0 \setminus A(s)$.
The success probability of this strategy is clearly at least $1/a_s$.
\end{proof}

\begin{definition}
\label{strong-G-optimal}
We will define a strong AMD code that meets the bound of Theorem \ref{strong-bound-G} 
with equality for every possible source $s$
to be \emph{G-optimal}. Again, ``G'' is used to indicate that guessing the most likely 
encoding is an optimal strategy. 
\end{definition}

\begin{corollary}
\label{strong-cor-G}
In any strong $(m,n,\he)$-AMD code, 
it holds that $\he \geq 1/a_{s'},$
where $a_{s'} = \min \{ a_s :s \in \S\}$. 

\end{corollary}
\begin{proof} 
The quantity $1/a_s$ is maximized when $a_{s}$ is minimized.
\end{proof}

In the case of a $k$-regular code, we have the following corollary.

\begin{corollary}
\label{strong-uniform-G}
In any $k$-regular strong $(m,n,\he)$-AMD code, 
it holds that $\he \geq 1/k$. 
\end{corollary}

We now have an easy proof of the following previously known bound.

\begin{theorem}\cite[Theorem 2.2]{CFP}
\label{strong-bound}
In any $k$-uniform, strong $(m,n,\he)$-AMD code, 
it holds that \[\he^2 \geq \frac{m-1}{n-1}.\]
\end{theorem}

\begin{proof}
From Corollary \ref{strong-bound-2}, we have
 \[\he \geq \frac{k(m-1)}{n-1}.\]
Furthermore, from Corollary \ref{strong-uniform-G}, we have 
$\he \geq {1}/{k}.$
Multiplying these two inequalities, we get
\[\he^2 \geq \frac{m-1}{n-1}.\]
\end{proof}

\noindent{\bf Remark:} We will determine in Theorem \ref{strong-bound-equiv} 
necessary and sufficient conditions for the bound of Theorem \ref{strong-bound} to be met with 
equality in all nontrivial cases, i.e., when $\he < 1$.

\subsection{Optimal Strong AMD Codes}
\label{strongoptimal.sec}

\subsubsection{R-Optimal Strong AMD Codes}

Suppose the source $s$ is fixed. Consider the strategy $g \mapsto g + \Delta$, where $\Delta \neq 0$.
Let $\epsilon_{\Delta,s}$ denote the success probability of this strategy.
Then it is clear that 
\begin{equation}
\label{goodhats.eq} 
\he_s = 
\max \{ \epsilon_{\Delta,s} : \Delta \neq 0\}.
\end{equation}

For any $\Delta \neq 0$, define 
\begin{equation}
\label{goods.eq} 
\good(\Delta,s) = \{ g : g \in A(s) \text{ and } g + \Delta \in A(s'), \text{where } s' \neq s\} .
\end{equation}
This is the same definition as (\ref{good.eq}), except that $s$ is now fixed.

\begin{lemma}
For any $\Delta \neq 0$,  it holds
that
\begin{equation}
\label{goodh2s.eq} 
\epsilon_{\Delta,s} = \sum_{g \in \good(\Delta,s)} \prob[E(s) = g].
\end{equation}
\end{lemma}

\begin{proof}
It is clear that 
\begin{eqnarray*} \epsilon_{\Delta,s} &=& \prob [ E(s)  \in \good(\Delta,s)]\\
& = & \sum_{g \in \good(\Delta,s)} \prob[E(s) = g].
\end{eqnarray*}
\end{proof}

\begin{theorem}
\label{strong-probs}
In any strong AMD code,
$\he_s = (a-a_s)/(n-1)$
if and only if $\epsilon_{\Delta,s} = (a-a_s)/(n-1)$ 
for all $\Delta \neq 0$.
\end{theorem}

\begin{proof}
Suppose we have an AMD code 
where  $\he_s = (a-a_s)/(n-1)$.
It is not hard to compute
\begin{eqnarray*}
\sum _{\Delta \neq 0} \epsilon_{\Delta,s} &=& \sum _{\Delta \neq 0} \sum_{g \in \good(\Delta,s)} \prob[E(s) = g]\\
&=&  \sum_{g \in A(s)}  \prob[E(s) = g] \times | \{\Delta :g \in \good(\Delta,s)  \}|  \\
&=& \sum_{g \in A(s)}  \prob[E(s) = g] \times  (a-a_s)    \\
&=& a-a_s.
\end{eqnarray*}
Therefore the average of the quantities $\epsilon_{\Delta,s}$ ($\Delta \neq 0$) is equal to 
$(a-a_s)/(n-1)$.
In order to have $\he_s = (a-a_s)/(n-1)$, it must be the case that
$\epsilon_{\Delta,s} = (a-a_s)/(n-1)$ for all $\Delta \neq 0$.
\end{proof}

\begin{theorem}
\label{GSEDF-construction-2}
Suppose there is an $(n,m;k_1, \dots , k_m;\lambda_1, \dots , \lambda_m)$%
-GSEDF.   
Then there is an R-optimal strong AMD code where $a = \sum_{i=1}^{m} k_i$.  
\end{theorem}

\begin{proof}
Suppose the GSEDF is given by $A_1, \dots , A_m$.
Let $\S = \{s_1, \dots , s_m\}$ be a set of $m$ sources. For $1 \leq i \leq m$, 
define $A(s_i) = A_i$, so $a_{s_i} = k_i$, and suppose
the encoding function $E(s_i)$ is equiprobable.
We show that $\epsilon_{\Delta,s_i} = (a-a_{s_i})/(n-1)$ for $1 \leq i \leq m$ and all $\Delta \neq 0$.
We have 
\begin{eqnarray*}
\epsilon_{\Delta,s_i} &=& 
\sum_{g \in \good(\Delta,s_i)} \prob[g]\\
 & = &   \frac{\lambda_i}{k_i}\\
& = &  \frac{a - k_i}{n-1} \quad \text{from (\ref{GESDF.eq})}\\
& = & \frac{a - a_{s_i}}{n-1}.
\end{eqnarray*}
\end{proof}

It is possible to prove a converse to Theorem \ref{GSEDF-construction-2} in the 
case where the AMD code has equiprobable encoding.

\begin{theorem}
\label{GSEDF-construction-converse}
Suppose there is an R-optimal strong AMD code with equiprobable encoding.
Then the sets $A(s)$ ($s \in \S$) form an  $(n,m;k_1, \dots , k_m;\lambda_1, \dots , \lambda_m)$%
-GSEDF.
\end{theorem}

\begin{proof}
Suppose the sources are denoted $\S = \{s_1, \dots , s_m\}$.
Fix a value $i$, $1 \leq i \leq m$ and let $\Delta \neq 0$.
We have 
\begin{eqnarray*}
\epsilon_{\Delta,s_i} &=& \frac{a - a_{s_i}}{n-1} \\
& = & \sum_{g \in \good(\Delta,s_i)} \prob[g]\\
& = &   \frac{|\good(\Delta,s_i)|}{a_{s_i}}.
\end{eqnarray*}
Therefore, for a fixed value $i$, we have
\[ |\good(\Delta,s_i)| = \frac{a_{s_i}(a - a_{s_i})}{n-1}\]
for all $\Delta \neq 0$. This says that
\[ \D(A(s_i), \G_0 \setminus A(s_i)) = \lambda_i (\G \setminus \{0\}),\]
where \[\lambda_i = \frac{a_{s_i}(a - a_{s_i})}{n-1}.\]
\end{proof}

\noindent{\bf Remark:}  The results we have proven in Theorems
\ref{GSEDF-construction-2} and \ref{GSEDF-construction-converse} establish a close
connection between R-optimal strong AMD codes and GSEDF.
In \cite{CFP}, similar results were proven, using the language
of differential structures, that showed the link
between (not necessarily optimal) strong AMD codes and BGSEDF.

\begin{figure}[tb]
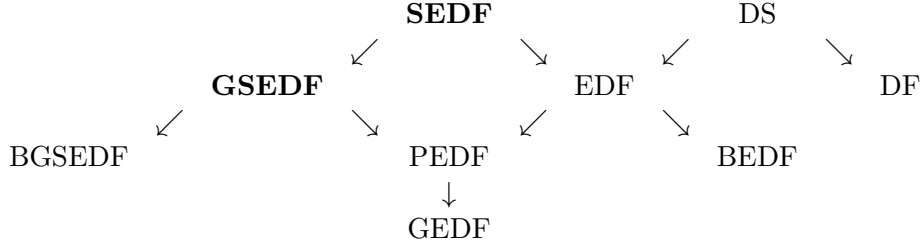

\[
\begin{array}{ccccccccccc}
& &  & & \mathbf{SEDF} & & & &  \mathrm{DS}\\
& &  & \swarrow &  & \searrow & & \swarrow &  & \searrow \\
& & \mathbf{GSEDF} & & & & \mathrm{EDF} & & & & \mathrm{DF}\\ 
& \swarrow &  & \searrow &  & \swarrow & & \searrow &  &  \\
\mathrm{BGSEDF} & &  & & \mathrm{PEDF} & & & &  \mathrm{BEDF}\\
& &  & & \downarrow \\
& &  & & \mathrm{GEDF}
\end{array}
\]
\caption{Difference
families that yield R-optimal strong AMD codes
(indicated in boldface type)}
\label{diff-strong-AMD-fig}
\end{figure}

\subsubsection{G-Optimal Strong AMD Codes}
\label{G-optimal-strong-sec}

Now we turn to G-optimality. We have the following characterization of G-optimal
strong AMD codes.

\begin{theorem}
\label{strong-equiv-2}
A G-optimal strong AMD code is equivalent to an $(n,m;k_1, \dots , k_m;1, \dots , 1)$-BGSEDF.
\end{theorem}

\begin{proof}
Suppose $A_1, \dots , A_m$ is an $(n,m;k_1, \dots , k_m;1, \dots , 1)$-BGSEDF.
Let $\S = \{s_1, \dots , s_m\}$ be a set of $m$ sources. For $1 \leq i \leq m$, define
an encoding function $E(s_i)$ which chooses an element of $A_i$ uniformly at random.
Let $1 \leq i \leq m$ and 
choose any $\Delta \in \G \setminus \{0\}$. 
Given that the source is $s_i$, the strategy $g \mapsto g + \Delta$ succeeds with 
probability at most $1/a_{s_i}$, since there is at most 
one $g \in A_i$ such that $g+ \Delta \in A_j$ and $j \neq i$.
Further, there exists a $\Delta \neq 0$ such that this strategy succeeds with probability
$1/a_{s_i}$.

Conversely, suppose we have a G-optimal strong AMD code. 
Let $s \in \S$.
From the proof of Theorem \ref{strong-bound-G}, it is easy to see that
all encodings of $s$ occur with the same probability $1/|A(s)|$.  Now
we claim that $\{ A(s) : s \in \S\}$ is an 
$(n,m;k_1, \dots , k_m;1, \dots , 1)$-BGSEDF. 
Suppose that there existed two different values $g,g' \in A_i$ such that $g+ \Delta \in A_j$,
$g'+ \Delta \in A_{j'}$ and $j,j' \neq i$.
It would then follow that $\he_s \geq 2/|A(s)|$, which is a contradiction.
\end{proof}


Now we show that  $k$-regular strong AMD codes  with $m \geq 3$ cannot be simultaneously R-optimal and G-optimal.

\begin{theorem}
\label{strong-non-existence}
There does not exist a strong 
AMD code with $m \geq 3$ and $\he < 1$ that is simultaneously R-optimal and G-optimal.
\end{theorem}

\begin{proof}
Since the code is G-optimal, it follows from Theorem \ref{strong-equiv-2} and its proof that
the code has equiprobable encoding and is derived from 
$(n,m;k_1, \dots , k_m;1, \dots , 1)$-BGSEDF. Now, since the code is R-optimal and it has equiprobable
encoding, Theorem \ref{GSEDF-construction-converse} shows that the code is derived from 
$(n,m;k_1, \dots , k_m;\lambda_1, \dots , \lambda_m)$-GSEDF.
Thus we have an $(n,m;k_1, \dots , k_m;1, \dots , 1)$-BGSEDF that is also an
$(n,m;k_1, \dots , k_m;\lambda_1, \dots , \lambda_m)$-GSEDF, so it must in fact be
an $(n,m;k_1, \dots , k_m;1, \dots , 1)$-GSEDF.
This implies that $k_i(a - k_i) = n-1$ for all $i$. Given $a$ and $n$,
the equation $x(a-x) = n-1$ has at most two distinct roots, and these roots
sum to $a$. Suppose that $k_i \neq k_j$ for some $i,j$. Then
$k_i + k_j = a$, which implies that $m=2$, a contradiction. Hence the code is
$k$-uniform and the GSEDF is in fact an $(n,m;k;1)$-SEDF.
Now Theorem \ref{SEDF-existence} implies that $k=1$ and
$n=m$. This code has $\he = 1$, so we are done.
\end{proof}

\begin{theorem}
\label{strong-bound-equiv}
There exists a $k$-uniform, strong $(m,n,\he)$-AMD code with 
$\he^2 = \frac{m-1}{n-1} < 1$ if and only if
$m=2$ and $n = k^2 + 1$.
\end{theorem}

\begin{proof}
Here we are considering $k$-uniform song AMD codes that are 
simultaneously R-optimal and G-optimal. From the proof of Theorem \ref{strong-non-existence}, 
we see that $m = 2$ and $k(a-k) = n-1$. Since $a = 2k$, we have $n = k^2+1$.
Conversely, if $m = 2$ and $n = k^2+1$, then Example \ref{exam10}
shows the existence of a $(k^2+1,2;k;1)$-SEDF. This yields a strong AMD code
with $\he = 1/k$, as desired.
\end{proof}

Figure \ref{diff-strong-AMD-fig-2} shows the types of difference
families that yield G-optimal strong AMD codes. The relevant difference
families are assumed to have $\lambda=1$ in this figure.

\begin{figure}[tb]
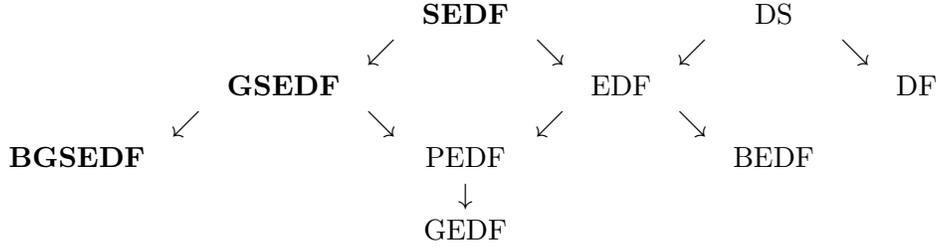

\[
\begin{array}{ccccccccccc}
& &  & & \mathbf{SEDF} & & & &  \mathrm{DS}\\
& &  & \swarrow &  & \searrow & & \swarrow &  & \searrow \\
& & \mathbf{GSEDF} & & & & \mathrm{EDF} & & & & \mathrm{DF}\\ 
& \swarrow &  & \searrow &  & \swarrow & & \searrow &  &  \\
\mathbf{BGSEDF} & &  & & \mathrm{PEDF} & & & &  \mathrm{BEDF}\\
& &  & & \downarrow \\
& &  & & \mathrm{GEDF}
\end{array}
\]
\caption{Difference
families with $\lambda = 1$ that yield G-optimal strong AMD codes
(indicated in boldface type)}
\label{diff-strong-AMD-fig-2}
\end{figure}

\section{Conclusion}
\label{conclusion.sec}

We have studied weak and strong AMD codes that provide optimal
protection against two specific adversarial substitution strategies. 
These codes are termed ``R-optimal'' and ``G-optimal''. We have
considered various types of generalized difference families
and determined when they yield R-optimal and/or G-optimal AMD codes.
As well, we have proven in certain situations that R-optimal and/or G-optimal AMD codes
imply the existence of the relevant difference families, thus
providing a combinatorial characterization of the AMD codes under consideration.

It is an interesting open problem to construct additional examples
of these generalized difference families. In particular, we ask if
there are any examples of strong external difference families
with $k > 1$ and $m > 2$. We are unaware of any such examples
at the present time.

\end{document}